\documentclass[12pt]{amsart}
\usepackage{amscd}
\usepackage[all]{xy}
\usepackage[notcite,notref]{showkeys}

\newcommand{\rat}{{\operatorname{rat}}}

\newcommand{\Hom}{\operatorname{Hom}}
\newcommand{\End}{\operatorname{End}}

\renewcommand{\Im}{\operatorname{Im}}

\newcommand{\C}{\mathbf{C}}
\renewcommand{\L}{\mathbf{L}}
\renewcommand{\P}{\mathbf{P}}

\newcommand{\Q}{\mathbf{Q}}
\newcommand{\Z}{\mathbf{Z}}
\newcommand{\un}{\mathbf{1}}

 \newcommand{\sT}{\mathcal{T}}
 \newcommand{\sC}{\mathcal{C}}
\newcommand{\sA}{\mathcal{A}}
\newcommand{\sM}{\mathcal{M}}

\newcommand{\sO}{\mathcal{O}}
\newcommand{\sH}{\mathcal{H}}
\newcommand{\sJ}{\mathcal{J}}

\newcommand{\sS}{\mathcal{S}}
 
 \newcommand{\sF}{\mathcal{F}}

 \numberwithin{equation}{section}

\theoremstyle{plain}
\newtheorem{thm}[equation]{Theorem}
\newtheorem{prop}[equation]{Proposition}
\newtheorem{lm}[equation]{Lemma}
\newtheorem{cor}[equation]{Corollary}
\newtheorem{conj}[equation]{Conjecture}

\theoremstyle{definition}

\newtheorem{ex}[equation]{Example}
\newtheorem{rk}[equation]{Remark}

\begin{document}

 \title[\ ] 
{On the  rationality and the finite dimensionality of a cubic fourfold}

  \author   { CLAUDIO PEDRINI}  \bigskip

 \maketitle
 
 \begin{abstract} Let $X$ be a cubic fourfold in $\P^5_{\C}$. We prove  that, assuming the Hodge conjecture for the product $S \times S$, where $S$ is a complex surface, and the finite dimensionality of  the Chow motive $h(S)$, there are at most a countable number  of  decomposable integral polarized Hodge structures, arising from the fibers of a family $f : \sS \to B$ of smooth projective surfaces. According to the results in [ABB] this is related to a conjecture proving the irrationality of a very general $X$. If $X$ is special, in the sense of B.Hasset,  and $F(X) \simeq S^{[2]}$, with $S$ a K3 surface associated to $X$, then we show that 
the Chow motive $h(X)$ contains as a direct summand a "transcendental motive" $t(X)$ such that $t(X)\simeq t_2(S)(1)$. The motive of $X$ is finite dimensional if and only if $S$ has a finite dimensional motive, in which case $t(X)$ is indecomposable. Similarly, if $X$ is very general and the motive $h(X)$ is finite dimensional, then $t(X)$ is indecomposable.\end{abstract}

\section {Introduction} 

Let  $\sM_{rat}(\C)$  be the   (covariant) category  of  Chow motives (with $\Q$-coefficients), whose objects are of the form $(X,p,n)$, where $X$ is a smooth projective variety over $\C$ of dimension $d$, $p$ is an idempotent in the ring $A^d(X \times X)= CH^d(X \times X)\otimes \Q $ and $n \in \Z$.  If $X$ and $Y$ are smooth  projective varieties  over $\C$, then the morphisms  $\Hom_{\sM_{rat}}(h(X),h(Y)) $ of their motives $h(X)$ and $h(Y)$ are given by correspondences in the Chow groups $A^*( X \times Y) =CH^*(X\times Y) \otimes \Q$. A similar definition holds for the category $\sM_{hom}(\C)$ of homological motives.The finite dimensionality of $h(X)$, as conjectured by S.Kimura, is known to hold  for curves, rational surfaces, surfaces with $p_g(X)=0$, which are not of general type and some 3-folds. In particular if $X$ is a cubic 3-fold in $\P^4_{\C}$,  then its motive is finite dimensional and  of abelian type, i.e. it lies in  the subcategory of $\sM_{rat}(\C)$ generated by the motives of abelian varieties, see[GG].\par
 If  $ d= \dim X \le 3 $,  then the finite dimensionality of $h(X)$ is a birational invariant ( see[GG, Lemma 7.1]), the reason being that in order  to make regular a birational map $\P^3 \to X$ one needs to blow up only  points and curves, whose motives are finite dimensional. Therefore every rational 3-fold has a finite dimensional motive.\par
 \noindent In the case   $ d= \dim X =  4$ the situation looks different. The finite dimensionality of $h(X)$ is not a priori
a birational invariant. In fact if   $r :  \P^4 \to X $ is a birational map then, by Hironaka theorem,  there is a composition
\begin{equation}  \tau = \tau_n \circ \tau_{n-1}\circ\cdots \circ \tau_1 :  \tilde X \to \P^4 \end{equation}
of monoidal   transformations $\tau_i : \tilde X_i \to \tilde X_{i-1}$ with smooth centers $C_{i-1}$ of dimensions $\le 2$, where $\tilde X_0 =\P^4$  and $\tilde X_n =\tilde X$,  such that\par
\noindent  $f =r\circ \tau :\tilde X \to X$ is a birational morphism. Therefore, by Manin's formula, we get 
\begin{equation} h(\tilde X) \simeq h(\P^4) \oplus \bigoplus_{ 1\le j \le n-1} h(C_j)(j)\end{equation}
and hence $h(\tilde X)$ is finite dimensional if all the motives $h(C_j)$ are finite dimensional, in which case also $h(X)$ is finite dimensional. Since the motives of points  and curves are finite dimensional the problem  is related to the finite dimensionality of those $C_j$ such that $dim \  C_j=2$.\par
Let $X$ be a cubic fourfold in $ \P^5_{\C}$. All known examples are rational while the very general one is conjecturally irrational. These examples of rational cubic fourfolds have an associated K3 surface.\par
The aim of this note is to show how  the rationality of  a  cubic fourfold and the finite dimensionality of its motive can be related to properties of  the motives of surfaces.\par
\noindent  In Sect.2 we recall  some  results by Y. Zarhin and  V.S. Kulikov on cubic fourfolds.  Zarhin in [Za] (see Theorem 2.1)  showed that there is  a restriction  on the class of surfaces that can appear by  resolving the indeterminacy of a rational map $r:  \P^4 \to X$, as in (1.1).  In particular, if  one assumes that a very general cubic fourfold $X$ is  rational , then among the surfaces $C_j$ appearing in (1.2) there must be surfaces $S$ that are neither rational nor K3, see Remark 2.2. Kulikov in [Ku]  (see Theorem 2.3 ) showed that a very general cubic fourfold in $\P^5_{\C}$  is not rational, if one assumes that the  polarized Hodge structure $\sT_S$ on the transcendental cycles of a smooth projective surface is integrally indecomposable. However  the above conjecture cannot be true because  the polarized Hodge structure of a Fermat sextic $S\subset \P^3_{\C}$ is  integrally decomposable, see [ABB].\par
\noindent In [ABB, Sect.4]  it is also  proved that the irrationality of a very general cubic fourfold $X$ will follow, if one can prove that there are only countably many integral polarized Hodge structures, arising from the transcendental lattice of a surface, which are decomposable. There  are countably many families $\sS_i \to B_i$ of complex surfaces, over  an irreducible base variety, such that every isomorphism class of a projective surface is represented  by a fiber of such a family. Therefore the irrationality of a very general cubic fourfold will follow if one could prove the following conjecture:\par
\begin {conj} Let $f  : \sS \to B$ be a smooth projective family of  complex surfaces over an irreducible base. There  are at most  countably many integral polarized Hodge structures $\sT_{S_b}$, arising from the fibers $S_b$ of $f$,  which are decomposable. \end{conj}
 \noindent In Sect. 3 we prove this conjecture for all families  $f : \sS \to B$ such that the  generic fiber has an indecomposable transcendental motive, the closed fibers $S_b$ have a finite dimensional motive and $S_b \times S_b$ satisfies the Hodge conjecture. We also compare our results with those appearing in [Gu] and in [ABB,4.7],  see Remark 3.7.\par
 \noindent In Sect 4 we consider the motive  $h(X)$ of a cubic fourfold  $X$  such that the  Fano variety of lines $F(X)$ is isomorphic to the Hilbert scheme$S^{[2]}$ of length-two subschemes of a  K3 surface  $S$ associated to $X$. We prove (Theorem 4.6) that $h(X)$ contains as a direct summand a motive $t(X)$ which is isomorphic to  $t_2(S)(1)$, where $t_2(S)$ is the transcendental motive of $S$. Examples of this type are given in  4.11(2) and 4.11 (3). As a corollary we get that $h(X)$ is finite dimensional if and only $h(S)$ is finite dimensional (see Cor.4.9), in which case the motive $t(X)$ is indecomposable. Using this result   
we show  that there is a family of K3 surfaces $S$,  which are double covers of $\P^2$ ramified along a sextic and are associated to a family $\sF$ of cubic fourfolds in $\P^5$,  such that the motives $h(S)$ are finite dimensional  and of abelian type, see 4.11(1).\par
\noindent Finally we prove ( Proposition 4.12) that for a very general cubic fourfold  $X$ with a finite dimensional motive $h(X)$  the transcendental motive $t(X)$ is indecomposable.

\section{The Theorems of Zarhin and Kulikov}

Let $X$ be cubic fourfold. A smooth surface $S$ is said to be disjoint from $X$ if one of the following conditions holds, where $\rho(S) =\dim_{\Q} (NS(S)_{\Q})$ ,  $\rho_2(X) = \dim_{\Q}  \bar A^2(X)$,  with  $\bar A^2(X) = Im (cl : A^2(X) \to H^4(X,\Q))$.\par
(i) $h^{2,0}(S)=0$;\par
(ii) $h^{1,1} (S) - \rho(S)< 21 -\rho_2(X)$;\par
(iii) $h^{2,0}(S) =1$ and $h^{1,1}(S) -\rho(S) \ne 21 -\rho_2(X)$.\par 

\noindent The following result  by Y.Zarhin gives  conditions for    
  the class of surfaces that can appear when  resolving the indeterminacy of a rational map $r:  \P^4 \to X$.
\begin {thm}(Zarhin) Let $X$ be a cubic fourfold in $\P^5_{\C}$ and let $r : \P^4 \to X$ be a rational map of finite degree. It is not possible to make $r$  regular by blowing up only points, curves , rational surfaces and surfaces  $S$ disjoint from $X$.\end{thm}
 \begin{rk} If  the cubic fourfold $X$ is very general then any algebraic surface $S \subset X$ is homologous to a complete intersection, and hence  $\rho_2(X)=1$.
Therefore  the family of surfaces  disjoint from $X$ contains all K3 surfaces $S \subset X$, because $h^{2,0}(S)=1$, $h^{1,1}(S)= 20$, $\rho(S)\ge1$ so that
$$h^{1,1}(S) -\rho(S) \le 19 < 20=21-\rho_2(X)$$ 
\noindent  The following example shows that, on the contrary,  when $X$ is not general and rational, then  the locus of indeterminacy of a rational map $r :  \P^4 \to X$ can consist of just one K3 surface $S$ which is not disjoint from $X$.\par
\noindent  Let  $X \subset \P^5_{\C}$  be a cubic fourfold containing two disjoint planes $P_1$ and $P_2$. Choose equations for $P_1$ and $P_2$ as follows 
$$P_1 = \{u=v=w=0\} , P_2 =\{x=y=z=0\} \subset \P^5$$
where  $\{u,v,w,x,y,z\} $ are homogeneous coordinates in $\P^5$. Fix forms $F_1,F_2 \in \C[u,v,w,x,y,z]$ of bidegree $(1,2)$ and $(2,1)$ in the variables $\{u,v,w\}$ and $\{x,y,z\}$. Then the cubic hypersurface of equation $F_1 +F_2 =0$ contains $P_1$ and $P_2$. As in  [ Has 2, 1.2]  one shows that cubic fourfolds of this type are rational.  There is well defined morphism
$$ \rho : ( P_1 \times P_2  - S)   \to X$$ 
which associates to every   $(p_1,p_2) \in P_1 \times P_2$, such that  the line $l(p_1,p_2)$ is not contained in $X$,  the third intersection with $X$. The map $\rho$ is birational because each point of $\P^5 -(P_1 \cup P_2)$ lies on a unique line joining the planes. Here $S$ is  the surface
$$S :  \{F_1(u,v,w,x,y,z) = F_2(u,v,w,x,y,z) =0\}.$$
which  is a K3 surface, a complete intersection  of bidegree $(1,2)$ and $(2,1)$, so that $\rho(S)=2$. Choose an  isomorphism    $\Phi : P_1 \times P_2 \simeq  \P^2 \times \P^2$, where $\{x,y,z\}$ are coordinates in the first copy and $\{u,v,w\}$  in the second one. Then  we get a birational map
$r : \P^4 \to X$ whose locus of indeterminacy is the K3 surface $S$,  which is not disjoint from  $X$. In fact we have $\rho_2 (X) \ge 4$, because $\bar A^2(X)$ contains  $\gamma^2$, with $\gamma$ the class of a hyperplane section, the classes of the two planes $P_1$ and $P_2 $ and the class of $S$ .  Therefore $h^{1,1}(S) -\rho(S)=18$, while $21 -\rho_2(X) \le 17$.\end{rk}
 A cubic fourfold $X$ has Hodge numbers $h^{4,0}=h^{0,4}=0$, $h^{3,1}=h^{1,3}=1$ and  $h^{2,2}=21$. $X$ satisfies the  integral Hodge conjecture  and hence 
$$ CH^2(X) = H^4(X,\Z) \cap H^{2,2}(X) \subset H^4(X,\Z)$$ 
  We also have $H^4(X,\Z) = \Z \gamma^2 \oplus P(X)_{\Z}$, where $\gamma \in H^2(X,Z) =CH^1(X)$ is the class of a hyperplane section and $P(X)_{\Z}$ denotes the  lattice of primitive cycles, i.e. 
 $H^4(X,\Z)_{prim}= P(X)_{\Z} = \ker ( p : H^4(X) \to H^6(X))$,  with $p(\alpha) = \alpha \cdot \gamma$. Poincare' duality induces a symmetric bilinear form
 $$  < \ , \ > : H^4(X)  \times H^4(X)\to \Z$$ 
 such that $<\gamma^2,\gamma^2> =3$. The restriction of  the bilinear form to $P(X)_{\Z}$ is non degenerate and hence yields a splitting
 $$CH^2(X) \simeq \Z \gamma^2 \oplus  CH^2(X)_0$$
 with $CH^2(X)_0 =\ CH^2(X) \cap P(X)_{\Z}$.\par
The integral polarized Hodge structure on the primitive cohomology of $X$, is denoted by
$$ \sH_X= (P(X),H_{prim}^{p,q}, < , >_{P(X)})$$
where $H_{prim}^{p,q}=H^{p,q}(X) \cap (H^4(X,\Z)_{prim} \otimes_{\Z} \C)$.
\noindent Let $T_X \subset P(X)_{\Z} \subset H^4(X,\Z)$ be the orthogonal complement of $CH^2(X)$ in $H^4(X,\Z)$ with respect to $ <  \ , \  >$,  i.e. the lattice of transcendental cycles 
$$H^4_{tr}(X,\Z)= T_X = \{h \in H^4(X)/ <h,\alpha>=0 , \alpha \in CH^2(X) \}.$$
If $X$ is very general then $\rho_2(X)=1$ and hence  the rank of $T_X$ is 22, because  it  is orthogonal, inside $H^4(X)$ of dimension 23,  to $CH^2(X)$ which has dimension 1. \par
\noindent  $T_X$ has a natural polarized Hodge structure 
$$\sT_X = \{T_X,H^{p,q}_T,<\ , \ >_T\}$$
with Hodge numbers $h^{3,1} =h^{1,3} =1$, $h^{2,2} =21-\rho_2(X)$ and $ <  \ , \  >_{T_X}$ as a polarization. It is an integral polarized Hodge substructure of the primitive cohomology. For a very general $X$ the polarized Hodge structure is irreducible.\par
\noindent A polarized Hodge structure
$\sT_X$ is the direct sum of two integral polarized  sub-Hodge structures $(T_{\Z,1},H^{p,q}_1, < , >_{T_1})$ and\par
\noindent   $(T_{\Z,2},H^{p,q}_2,  < , >_{T_2})$ if 
$$\sT_X \simeq   (T_{\Z,1} \oplus T_{\Z,2},H^{p,q}_1 \oplus  H^{p,q}_2, < , >_{T_1}\oplus < , >_{T_2})$$
An integral polarized Hodge structure is indecomposable if it is not a direct sum of two non trivial integral polarized Hodge structures.\par
\noindent   If $S$ is a  complex surface  then the polarized Hodge structure $\sT_S$ on the transcendental cycles $H^2_{tr}(S) =T_S$ is given by 
$$\sT_S =(T_S,H^{p,q}_{T_S}, <,>_{T_S})$$
where $H^{p,q}_{T_S}= (T_S \otimes \C)\cap H^{p.q}(S)$ and $< , >_{T_S}$ is induced by the intersection form on $H^2(S,\Z)$. The polarized Hodge structure  $\sT_S$ is a birational invariant. 
If $p_g(S) \le1$ then $\sT_S$ is integrally   indecomposable.\par
Let  $X$ be a rational cubic fourfold, $r : \P^4 \to X$ a birational map and let $f = r  \circ \tau : \tilde X \to X$ be a birational morphism, as (1.1). Then we have a decomposition of the polarized Hodge structures
$$\sH_X =\tau^*(\sH_{\P^4}) \oplus \bigoplus_{ 0 \le i\le n-1} \sH_i$$ 
where $\sH_i$ is the contribution in $\sH_X$ of the $(i+1)$-th monoidal transformation $\tau_{i+1}$.\par
\noindent V.S. Kulikov  in [Ku, Lemma 2]  proved the following result in the case $X$ is a very general cubic fourfold.
\begin{thm}(Kulikov) Let $X$ be a  very general smooth cubic fourfold which is rational and let $f =r\circ \tau : \tilde X \to X$  be as in (1.1). Then there is an index $i_0 \in \{1,\cdots,n\}$ such that $C_{i_0}$ is a smooth surface $S$ and  the polarized Hodge structure $\sT_S$ on the transcendental cycles  $T_S= H^2_{tr}(S)$  can be decomposed into the direct sum of polarized Hodge structures 
$$\sT_S =-f^*(\sT_X)(1)  \oplus  \tilde \sT_S $$ 
\end{thm}
 The above theorem  is used  in [Ku] to prove that a  very general cubic fourfold is not rational, if one assumes the following conjecture: the integral polarized Hodge structure $\sT_S$ on the transcendental cycles on a smooth projective surface is non decomposable. However  in [ABB] it is proved that the above conjecture cannot be true because  the integral polarized Hodge structure of a Fermat sextic $S\subset \P^3_{\C}$ is decomposable. Note that, according to (2.1),  among the surfaces $C_1,\cdots C_n$  in Theorem 2.3 there are surfaces  $S$ which are not rational, not K3 and  also not $\rho$-maximal, because if  $h^{1,1}(S) =\rho(S)$, then $S$ is disjoint from $X$.

 \section{Hodge  conjecture and the rationality of a cubic fourfold} 
  
  In this  section we prove  that, assuming the Hodge conjecture for the product $S \times S$, where $S$ is a complex surface, and the finite dimensionality of $h(S)$,  there are at most a countable number  of  decomposable integral polarized Hodge structures in a family $f : \sS \to B$ of smooth projective surfaces, such that the generic fiber has an indecomposable transcendental motive\par
 \noindent We first recall  from [KMP, 7.2.2]) some results on the Chow-K\"unneth decomposition of the motive of surface.\par
Let $S$ be a smooth projective surface  over a field $k$. The motive $h(S) \in \sM_{rat}(k)$ has a refined Chow-K\"unneth decomposition  
$$h(S)=\sum_{0 \le i\le 4}h_i(S)=\sum_{ 0\le i \le 4} (X,\pi_i)$$
where   $h_2(S) =h^{alg}_2(S) +t_2(S)$, with  $\pi_2 =\pi^{alg}_2 +\pi^{tr}_2$,  $t_2(S)= (S,\pi^{tr}_2)$ and $h^{alg}_2 =(S,\pi^{alg}_2)\simeq \mathbf{L}^{\oplus \rho(S)}$. 
 Here  $ \rho= \rho(S)$  is  the  rank of $NS(S)$.  The transcendental motive $t_2(S)$ is independent of the Chow-K\"unneth decomposition and is a birational invariant for $S$.
 We also have 
 $$H^i(t_2(S))=0  \   for \   i \ne 2    \  ;   \  H^2(t_2(S)) =\pi^{tr}_2 H^2(S,\Q) =H^2_{tr}(S,\Q),$$  
  $$A_i(t_2(S))=\pi^{tr}_2 A_i(S)=0    \  for  \    i \ne 0  \  ; \  A_0(t_2(S)) = T(S),$$  
 where $T(S)$ is the Albanese kernel. The map $\Phi: A^2(S \times S) \to \End_{\sM_{rat}}(t_2(S)$,
defined by $\Phi(\Gamma) =\pi^{tr}_2 \circ \Gamma \circ \pi^{tr}_2 $,  induces an isomorphism (see [KMP, 7.4.3]) 
\begin{equation}  {A^2(S \times S) \over \sJ(S)} \simeq \End_{\sM_{rat}}(t_2(S))\end{equation}
where $\sJ(S)$ is the ideal in $A^2(S \times S)$ generated by correspondences   whose support is contained either in $V\times S$ or in $S\times W$, with $V,W$ proper closed subsets of $S$.
 \begin{lm} Let $S$ be a smooth projective complex surface,with $p_g(S) \ge2$, such that the polarized Hodge structure  $\sT_{S}$ is integrally decomposable. Assume that\par
 (i) The variety $S \times S$ satisfies  the Hodge conjecture;\par
 (ii) The motive $h(S)$ is finite dimensional.\par
 \noindent Then the transcendental motive $t_2(S)$ is decomposable.\end{lm}  
 \begin{proof} Let $\sM_{hom}(\C)$ be the category of homological motives and let  $PHS_{\Q}$ be the  category of  polarized Hodge structures, which is abelian and semisimple. There is a Hodge realization functor 
 $$H_{Hodge} : \sM_{rat} (\C)\to PHS_{\Q}$$ 
 which factors trough $\sM_{hom}(\C)$. Let  $h(S) =\sum_{0\le i\le 4} h_i(S)$ be  a  refined Chow-K\"unneth decomposition with $h_2(S) =h^{alg}_2(S) + t_2(S)$. In the corresponding decomposition for the homological motive $h_{hom}(S)$ one has $h_{2,hom}(S) =t^{hom}_2(S) \oplus \L^{\oplus \rho(S)}$, where $t^{hom}_2(S) =(S ,\pi^{tr}_{2,hom}) $ is the image of $t_2(S)$in $\sM_{hom}$.\par
 \noindent  The Hodge structure  $\sT_{S }$, being decomposable, is the direct sum of two non-trivial polarized Hodge structures. Let  $L \subset H^2_{tr} (S, Q)$ be a non-trivial Hodge substructure such that 
$$ H^2_{tr}(S,\Q)) =L \oplus L'$$ 
where $L'$ is the orthogonal complement of $L$ with respect to the intersection pairing $Q=< \  , \  >_{H^2_{tr}(S)}$. $L$ defines  a morphism of polarized Hodge structure
$$\Phi_L : H^2_{tr}(S,\Q) \to H^2_{tr}(S,\Q)$$
such that $\ker \Phi =L'$ and $\Im \Phi =L$.  We also have $\Phi_L \circ \Phi_L =\Phi_L$.  The K\"unneth decomposition and Poincare' duality  give isomorphisms
$$H^4( S  \times S ,\Q)= \oplus_{p+q=4} (H^p(S,\Q) \otimes H^q(S ,\Q)) \simeq \oplus_p \Hom (H^p(S ),H^p(S ))$$ 
and these isomorphisms are isomorphisms of Hodge structures. The class $\Gamma_L \in H^4(S  \times S ,\Q)$, which corresponds to $\Phi_L$, is a Hodge class, because $\Phi_L $ is a morphism of Hodge structures, see [Vois 5, 1.4]. Therefore  $\Gamma_L \in H^4(S  \times S ,\Q)\cap H^{2,2}(S \times S)$ and hence, by  the Hodge conjecture, $\Gamma_L$ is the class of an algebraic cycle, i.e
 $\Gamma_L \in \bar A^2(S  \times S )$, with $\bar A^2(S  \times S ) = \Im (A^2(S \times S)  \to H^4(S \times S,\Q))$. Then  $\Gamma_L \in \End_{\sM_{hom}}(t^{hom}_2(S))$ and   $\Gamma_L$  is  a projector  different from 0 and from the identity. It follows that $\Gamma_L$ defines a non-trivial submotive  of $t^{hom}_2(S) =(S ,\pi^{tr}_{2,hom}) $. Therefore there exist a projector $\tilde \pi$ such that
$$ \pi^{tr}_{2,hom}= \Gamma_L+\tilde\pi$$
with $\tilde \pi \ne \pi^{tr}_{2,hom}$. From the finite dimensionality of $h(S)$ we get that a similar decomposition also  holds in $\sM_{rat}(\C)$ and hence
 $$t_2(S) =(S,\pi_1) +(S ,\pi_2)$$
 where $(S,\pi_i)$ is a non-trivial submotive of $t_2(S)$. Therefore  $t_2(S)$ is decomposable in $\sM_{rat}(\C)$
 \end{proof} 

\begin{ex} ($\rho$-maximal surfaces) In [Beau] several examples are given of $\rho$-maximal surfaces $S$ such that $h(S)$ is finite dimensional and $S \times S$ satisfies the Hodge conjecture. A complex surface $S$ is $\rho$-maximal if $H^{1,1}(S) =NS(S) \otimes \C$ and hence the Picard number $\rho(S)$ equals $h^{1,1}(S)$. Then the subspace $H^{2,0} \oplus H^{0,2}$ of $H^2(S,\C)$ is defined over $\Q$ (see [Beau, Prop.1] ) and it is the orthogonal complement of $H^{1,1}(S)$. Therefore $T_S \otimes \C =H^{2,0} \oplus H^{0,2}$ and hence the polarized Hodge structure $\sT_S$ decomposes over $\Q$. These examples include the Fermat sextic surface in $\P^3$ and the surface $S =C_6 \times C_6 \subset \P^4$, where $C_6$ is the plane sextic $x^6_0 +x^6_1 +x^6_2=0$. By Lemma 3.2 the motive $t_2(S)$ is decomposable. In the case of the Fermat sextic surface  also the integral  polarized  Hodge structure on the  transcendental part $H^2_{tr} (S,\Z)$ of the cohomology is decomposable, see [ABB,Theorem 1.2 ].\end{ex}
The following Lemma appears in [Gu,Lemma 13].
 \begin{lm} Let $S$ be a smooth projective surface defined over a field $k$ and let $k \subset L$ be a field extension. If  the motive $t_2(S_L)$ is indecomposable then also $t_2(S)$ is indecomposable.\end{lm}
\begin{proof} Suppose that the motive $t_2(S)$ splits into two non-trivial direct summands $T$ and $T' $ in the category $\sM_{rat}(k)$, with $T =(S,\pi)$ and $T ' =(S,\pi')$. Then $\pi^{tr}_2 =\pi +\pi'$, where $\pi$ and $\pi'$ are    idempotents in $A^2(S \times S)$.  Extending scalars from $k $ to $L$ we get a decomposition of $t_2(S_L)$ into the motives $T_L$ and $T'_L$ , by means of the  projectors
$\pi_L$ and $\pi'_L$. Since the motive $t_2(S_L)$ is  indecomposable we get that either $\pi_L$ or $\pi'_L$ is zero. Suppose   $\pi_L=0$. Then, by the results in [Gil], $\pi_L$ is nilpotent and hence it  vanishes. Therefore $T=0$ which is a contradiction because $T$ is  a non-trivial summand of $t_2(S)$.\end{proof}

\begin{thm} Let $f : \sS \to B$ be a smooth projective family of complex surfaces with $B$ an irreducible variety of dimension $n$. Let  $K$ be the algebraic closure of the function field of $B$ and let $S_K$ be the generic fibre  of $f$ over $K$. Assume that the transcendental motive $t_2(S_K)$ is indecomposable. Then there exists at most a countable number of $b \in B$ such that $t_2(S_b)$ is  decomposable. \end{thm}
\begin {proof} By [Vial 4, Lemma 2.1] there exist a subset $U \subset B(\C)$, which is a countable intersection  of Zariski open subsets $U_i$, such that, for each point  $b \in U$ there is an isomorphism between $\C$ and $K$ that yields an isomorphism between  the scheme $S_b$ over $\C$ and the scheme $S_K$ over $K$. Here $U_i = (B_0 - Z_i)(\C)$ where $B_0$ is defined over a countable subfield $k \subset \C$ and $Z_i$ is a proper subvariety  of $B_0$. Since the Chow groups of a variety $X$ over a field only depend on $X$ as a scheme we get an isomorphism 
$$CH^2(S_b \times S_b) \simeq CH^2(S_K \times S_K)$$
which, by applying the isomorphism in (3.1), induces an isomorphism between the $\Q$-algebras $\End_{\sM_{rat}}(t_2(S_b))$ and $\End_{\sM_{rat}}(t_2(S_K))$. Since $t_2(S_K)$ is indecomposable 
$\End_{\sM_{rat}}t_2(S_K) \simeq \Q$. Therefore\par
\noindent  $\End_{\sM_{rat}}(t_2(S_b))\simeq \Q$  and hence $t_2(S_b)$ is indecomposable, for all $b \in U (\C)=\bigcap_i U_i(\C) \subset B(\C)$.\par
\noindent  We claim that there exists at most  a countable number of $b \in B(\C)$ such that the corresponding  fibre $S_b$ has a decomposable transcendental motive.\par
\noindent  If $t_2(S_b)$ is decomposable then  $b \notin U(\C)$ and hence  there exists an  index $i_0$ such that $b \in Z_{i_0}(\C)$, with $Z_{i_0}$ a subvariety of $B_0$. Let  $Z =Z_{i_0} $ with $ \dim \  Z =m <n$. We proceed by induction on $m$.\par
\noindent  If $m=0$ then $Z(\C)$ is just a finite collection of points in $B(\C)$ and hence the corresponding fibers are a finite set. Assume that, if  $Y \subset B_0$ with $\dim Y \le m-1$, there are only a countable number of $b \in Y(\C)$ such that $t_2(S_b)$ is decomposable. Let  $Z \subset B_0$ of dimension $m$ and let $z \in Z(\C)$ be  such that  $t_2(S_z)$ is decomposable. Since there are only a finite number of irreducible components of $Z$ we may as well assume that $Z$ is irreducible. Let $\eta_Z$ be the generic point of $Z$ and let $F = k(\eta_Z)$ be the  function field of $Z$, which is a finitely generated extension of the countable subfield $k$ of $\C$. Choose an embedding  of  the algebraic closure $L=\bar  F$ into $\C$. Then $L \subset K$. By Lemma 3.4 $t_2(S_L)$ is indecomposable, because $t_2(S_K)$ is indecomposable. Therefore we may apply Lemma 2.1 in [Vial 4] to the family $f_Z : \sS_Z \to Z$,  where $\sS_Z =\sS \times_B Z$, whose generic fiber over the algebraic closure of $k(Z)$ is $S_L$. It follows that
$t_2(S_z)$ is indecomposable for  all points $z \in Z(\C)$ in a countable intersection $V= \cap_j (Z -T_j)(\C)$, where $T_j $ is a proper subvariety of $Z$. Therefore, if $t_2(S_z)$ is decomposable,  then there exists 
an index $j_0$ such that  $z \in T_{j_0}(\C)$. Since $\dim T_{j_0} \le m-1$, by the induction hypothesis, there are at most a countable number of  points $z \in T_{j_0} (\C)\subset B(\C)$ such that $t_2(S_z)$  is decomposable. Therefore   there are at most a countable  number of points in $ z \in Z(\C)$ such that the corresponding fiber $S_z$ has a decomposable transcendental motive.\end{proof}
\begin {cor} Let $f : \sS \to B$ as in Theorem 3.5 and assume furthermore that for all surfaces in the family  the motive $h(S)$ is finite dimensional and $S\times S$ satisfies  the Hodge conjecture.
 Then the integral polarized Hodge structure  $\sT_{\sS_b}$ arising from the fibers $\sS_b$ which are decomposable are countably many.\end{cor}
 \begin {proof}  Let $b \in B$ be such that the  polarized Hodge structure  $\sT_{S_b}$ is integrally decomposable. By Lemma 3.2  the transcendental motive $t_2(S_b)$ is decomposable. From Theorem 3.5 we get that  $b$ belongs to a countable subset of $B(\C)$
\end{proof}
\begin {rk}(Integral decomposability) In [Gu, Thm.22] it is proven, using the Theorem of Kulikov, that a very general cubic fourfold  $X$ is not rational if one assumes that, for every smooth projective surface $S$,  the transcendental motive $t_2(S)$ is  integrally indecomposable and finite dimensional. Since the motive $t_2(S)$ lives in the category of $\sM_{rat}$ with rational coefficients it is not a priori clear what it means to be  integrally  decomposable. According to the definition given in [Gu] the motive  $t_2(S)$ decomposes  integrally if the identity of the ring $R=CH^2(S \times S)/\sJ(S)$ is the sum of two orthogonal non-torsion idempotents, i.e. if there exist projectors  $\Pi_1 , \Pi_2  $ such that the class $[\Delta_S]$ in $R$ decomposes as $[\Delta_S] = [\Pi_1]+ [\Pi_2]$. After tensoring with $\Q$ the ring $R\otimes \Q$  becomes isomorphic to $End_{\sM_{rat}}(t_2(S))$, as in (3.1), and the image of  $[\Delta_S]$ is the identity of $t_2(S) =(S,\pi^{tr}_2)$. Therefore $\pi^{tr}_2 = \Phi (\Pi_1) +\Phi (\Pi_2)$, with $\Phi : A^2(S \times S) \to End_{\sM_{rat}}(t_2(S))$.\par
\noindent  Let  $S =C_6 \times C_6 \subset \P^4$, where $C_6$ is the Fermat sextic in $\P^2_{\C}$. The surface $S$ is $\rho$-maximal and the  transcendental motive $t_2(S)$ is finite dimensional  and decomposable in $\sM_{rat}$, see Ex.3.3. In  [Gu,Sect 5]  it is proven that $t_2(S)$  is not integrally decomposable.\par
\noindent In [ABB, 4.7] it is noted that the irrationality of a very general cubic fourfold will follow if every complex surface $S$, such that the  integral polarized Hodge structure  $\sT_S$ is decomposable, has maximal rank. According to Theorem 2.1, among the surfaces  $S_i$, with  $ 1\le i\le n$, that  can appear when  resolving the indeterminacy of a rational map $r:  \P^4 \to X$, for  a cubic fourfold $X$, there are surfaces which are not $\rho$-maximal, because these are all disjoint from $X$. On the other hand, by  Theorem 2.3, there is an index $i_o$ such that $\sT_{S_{i_0}}$ is decomposable.
\end{rk}

\section {The motive of cubic fourfold}
Let $X$ be a  complex  Fano  3-fold. By the results  in [GG]  its motive  $h(X)$ has  the following 
Chow-K\"unneth decomposition
 $$h(X) =\un \oplus   \L^{\oplus b}\oplus N \oplus (\L^2)^{\oplus b}  \oplus \L^3$$
where $b =dim NS(X)_{\Q}$ and $N =h_1(J) \otimes \L$, with $J=J^2(X)$ the intermediate Jacobian of $X$. Therefore $h(X)$ is finite dimensional and lies in the subcategory of $\sM_{rat}(\C)$ generated by the motives of Abelian varieties. In particular the motive of cubic 3-fold  in $\P^4_{\C}$ is finite dimensional.\par
 \noindent The  above result  is based on the fact that all Chow groups of a cubic 3-fold are representable. This is not the case if $X$ is  a cubic fourfold in $\P^5_{\C}$, because not all the Chow groups of $X$ are representable. Indeed if, for a smooth projective variety, all Chow groups are representable, then the Hodge numbers $h^{p,q}$ vanish, whenever $\vert p -q \vert >1$ (see [Vial 2. Thm.4])   and this is not the case for a cubic fourfold $X$,  because $h^{3,1}(X)=h^{1,3}(X)=1$.\par
In this section we show that, for a cubic fourfold $X$ such that its Fano variety $F(X)$ is isomorphic to $S^{[2]}$, with $S$ a K3 surface, then $h(X)$ is finite dimensional if and only if $h(S)$ is finite dimensional.
Note that, by a result of R.Laterveer in [Lat ],  if $h(X)$ is finite dimensional then also $h(F(X))$ is finite dimensional.\par 
 \noindent Every cubic fourfold $X$ is rationally connected and hence $CH_0(X) \simeq \Z$. Rational, algebraic and homological equivalences all coincide for cycles of codimension 2 on $X$. Hence the cycle map 
$CH^2(X) \to H^4(X,\Z)$ is  injective and $A^2(X)=CH^2(X) \otimes \Q$ is a sub vector space of dimension $\rho_2(X)$ of $H^4(X,\Q)$.  By  the results in [TZ] we  have  $A_1(X)_{hom} = A_1(X)_{alg}$. Moreover homological equivalence and numerical equivalence coincide for algebraic cycles on $X$, because the standard  conjecture $D(X)$ holds true. Therefore $A_1(X)_{hom}=A_1(X)_{num}$.\par
 \noindent A cubic fourfold  $X$ has no odd cohomology and  $H^2(X,\Q) \simeq NS(X)_{\Q} \simeq A^1(X)$, because  $H^1(X ,\sO_X) =H^2(X,\sO_X)=0$. Let $\gamma\in A^1(X)$  be the class of  a hyperplane section. Then $H^2(X,\Q)= A^1(X) \simeq \Q \gamma$ and  $H^6(X,\Q)= \Q[\gamma^2/3]$. Here   $< \gamma^2,\gamma^2>= \gamma^4 =3$, where   $ < \ , \ >$ is the intersection form on $H^4(X,\Q)$.\par
\noindent Let  $\pi_0 =[X \times P_0], \pi_8 = [P_0 \times X]$, where $P_0$ is a closed point and 
$\pi_2 =(1/3) (\gamma^3  \times \gamma)$, $\pi_6 =\pi^t_2 = (1/3)( \gamma \times \gamma^3)$. Then 
$$h(X) \simeq \un \oplus h_2(X) \oplus h_4(X) \oplus h_6(X) \oplus \L^4$$
where  $\un \simeq(X, \pi_0)$, $\L^4 \simeq (X, \pi_8)$, $h_2(X)=(X,\pi_2)$, $h_6(X) = (X,\pi_6)$ and $h_4(X) =(X,\pi_4)$, with $\pi_4 =\Delta_X -\pi_0-  \pi_2 -\pi_6 -\pi_8$. The above decomposition of the motive $h(X)$ is in fact integral, because
$$\gamma^3 =3\vert l \vert$$
for a line $l \in F(X)$, see [SV, Lemma A3].\par
 \noindent Let  $\rho_2$ be the dimension of $A^2(X)$ and let $\{D_1, D_2\cdots,D_{\rho_2} \}$  be a $\Q$-basis . We set 
 $$\pi_{4,i} = {[D_i \times D_i]\over <D_i,D_i>},$$
 for $i =1,\cdots, \rho_2$. Then $\pi_{4,i}$ are idempotents and the motive $M_i =(X,\pi_{4,i},0)$ is isomorphic to the Lefschetz motive $\L$ for every $i$. Let $h^{alg}_4(X)$ be the motive $(X,\pi^{alg}_4)$ where
 $\pi^{alg}_4 = \sum_{ 1 \le i \le \rho_2}\pi_{4,i}$. Then  the motive  $h_4(X)$  splits as follows 
$$h_4(X) =h^{alg}_4(X)  \oplus h^{tr}_4(X),$$
where  $h^{alg}_4(X) \simeq \L^{\oplus \rho_2}$ and $h^{tr}_4(X) = (X,\pi^{tr}_4)$, with $\pi^{tr}_4 =\pi_4 - \pi^{alg}_4$.\par 
\noindent Therefore we  get the following  Chow-K\"unneth decomposition of the motive $h(X)$

  \begin {equation} h(X) =  \un +h_2(X) + \L^{\oplus \rho_2} +t(X) +h_6(X)+ \L^4 \end{equation}

\noindent where $t(X) =h^{tr}_4(X)$. The motives $h_2(X)$ and $h_6(X)$ are finite dimensional (see [Ki, 8.4.5]) and hence in  the decomposition  (4.1) all motives, but possibly $t(X)$, are finite dimensional. Therefore  the motive $h(X)$ is finite dimensional if and only if $t(X)$ is  evenly finite dimensional.\par \noindent Let  $\sM^o_{rat}(k)$ the category of birational motives  (see [KMP, 7.5]). Writing  $\bar h(X)$ for the birational motive of a smooth
projective variety $X$ of dimension $d$, and $\Hom_{bir}$ for morphisms
in $\sM^o_{rat}(k)$, the fundamental formula is:
$$A_0(X_{k(Y)})=\varinjlim_{U\subset Y} A^d (U \times X)
                 \cong \Hom_{bir}(\bar h(Y), \bar h(X)).$$
If $X$ is a cubic fourfold then $\bar h(X)=\un$ in $\sM^o_{rat}(\C)$, because $X$ is unirational, see [Kahn,7.3]. From the decomposition in (4.1) we get
$\bar t(X)=0$, because the Lefschetz motive $\L$ goes to 0 in $\sM^o_{rat}$. Therefore the identity map  of $t(X)$ factors trough a motive of the form $M(1)$, see [KMP, 7.5.3]. More precisely, by  [Vial 3 2.2],  the motive $t(X) $ is a direct summand of $h(Z)(1)$ for some surface $Z$ and hence $t(X)$ is finite dimensional if $h(Z)$ is finite dimensional.
 \begin {lm}Let $X$ be a cubic fourfold and let $t(X)$ be the transcendental motive in the Chow-K\"unneth decomposition (4.1). Then $A^i(t(X)) = 0$ for $i \ne 3$ and $A^3(t(X))=A_1(X)_{hom}$
  \end{lm}
\begin{proof} The cubic fourfold $X$ is rationally connected and hence $A^4(X)=A_0(X)=0$ that implies $A^4(t(X))=0 $. Also from the Chow-K\"unneth decomposition in (4.1) we get $A^0(t(X)) =0$ and  
$A^1(t(X))= \pi^{tr}_4 A^1(X)=0$, because $A^1(h_2(X))= \pi_2A^1(X) =A^1(X)$\par
\noindent We first  show that $A^2(t(X))=0$. Let $\alpha \in A^2(X)$. Then
$$\pi^{tr}_4 (\alpha) = \alpha -  \pi_0(\alpha) -\pi_2(\alpha) -\pi^{alg}_4(\alpha) - \pi_6(\alpha) - \pi_8(\alpha),$$
where $\pi_0(\alpha) = \pi_8(\alpha)=0$. We also have $\pi_2(\alpha) = (1/3)(p_2)_*(p^*_1(\alpha)\cdot [\gamma^3 \times\gamma])=0$ and $\pi_6(\alpha) = (1/3)(p_2)_*(p^*_1(\alpha)\cdot [\gamma^1 \times\gamma^3])=0$. Let $\alpha \in  A^2(X)$ and let $\pi^{alg}_4 = \sum_{1 \le i \le \rho_2} \pi_{4,i} $, with $\pi_{4,i} = [D_i \times D_i ] / <D_i,D_i>$, 
where $\{D_1,\cdots, D_{\rho_2}\}$ is  a $\Q$-basis for $ A^2(X)$. Let  $\alpha =\sum_{ 1\le i \le \rho_2} m_i D_i$,with $m_i \in \Q$. Then $\pi^{alg}_4(\alpha) =\alpha$, because $(\pi_{4,i})_*(D_i)=  D_i $.
From the decomposition in (4.1) we get $\pi^{tr}_4 (\alpha)  =\alpha -  \pi^{alg}_4 (\alpha) =0 $ and hence
$$A^2(t(X))= (\pi^{tr}_4)_*A^2(X)=0.$$
\noindent Therefore we are left to show that $A_1(t(X))=A_1(X)_{hom}$. Let $\beta \in A_1(X) =A^3(X)$. From the Chow-K\"unneth decomposition in (4.1) we get
$$\pi^{tr}_4 (\beta) = \beta -  \pi_0(\beta) -\pi_2(\beta) -\pi^{alg}_4(\beta) - \pi_6(\beta) - \pi_8(\beta),$$
where $\pi_0(\beta) = \pi^{alg}_4(\beta) =\pi_8(\beta)=0$. We also have $\pi_2(\beta) = 0$ because $ \pi_2 =(1/3)(\gamma^3 \times \gamma)$. Therefore
$$\pi^{tr}_4(\beta) =\beta - \pi_6(\beta) = \beta -(1/3)(\gamma \times\gamma^3)_*(\beta) = \beta -(1/3) (\beta \cdot \gamma)\gamma^3 \in A^3(X) $$
and hence 
$$(\pi^{tr}_4(\beta)\cdot \gamma) = (\beta \cdot \gamma -  (1/3) (\beta \cdot  \gamma)\gamma^3 ) \cdot \gamma=0$$ because $\gamma^4 =3$.
 Since $\gamma$ is a generator of $A^1(X)$ it follows that  the cycle   $\pi^{tr}_4(\beta)$ is  numerically trivial. Therefore we get
$$A_1(t(X))=\pi^{tr}_4A_1(X)=A_1(X)_{num}=A_1(X)_{hom}$$.\par
\end{proof}
The following Lemma comes from the results in [Has 3, Thm. 3.18] and [GG, Lemma 1].
\begin{lm} Let $f : M \to N$ be a morphism of motives in $\sM_{rat}(\C)$ such that $f_*:  A^i(M) \to A^i(N)$ is an isomorphism for all $i \ge0$. Then $f $ is an isomorphism.\end{lm}
\begin{proof} Let $M=(X,p,m)$  and  $N =(Y,q,n)$ and let  $k \subset \C$ be a field of definition of  $f$, which is finitely generated . Then $\Omega=\C$ is a universal domain over $k$. By [Has 3, Thm. 3.18] the map $f$ has a right inverse, because the map $f_*: A^i(M)\to A^i(N)$ is surjective.  Let  $g : N \to M$ be such that $f \circ g =id_N$.  Then $g$ has an image $T$ which  is a direct factor of $M$ and hence $f$ induces an isomorphism  of motives in $\sM_{rat}(\C)$ 
$$f :  M \simeq N \oplus T$$
From the isomorphism  $A^1(M)\simeq A^i(N$, for all  $i \ge 0$, we get $A^i(T)=0$ and hence $T=0$, by [GG, Lemma1].\end{proof} 
Let $X$ be a cubic fourfold and let $F(X)=F $ be its Fano variety of lines , which is a smooth fourfold. Let
\begin {equation}\CD  P@>{q}>> X \\
@V{p}VV   @.         \\
F
\endCD \end{equation}
be the incidence diagram, where $P \subset X \times F$ is the universal line over $X$.  let   $ p_*q^* : H^4(X,\Q) \to H^2(F ,\Q)$ be the Abel-Jacobi map.  Let $\alpha_1,\cdots ,\alpha_{23}$ be a $Q$-basis of the vector space $H^4(X,\Q)$. Then, by a result of Beauville -Donagi, $\tilde \alpha_i =p_*q^*(\alpha_i)$ form a basis of $H^2(F,\Q)$. Let $q_F$ be the Beauville-Bogomolov bilinear form on $H^2(F,\Q)$ defined as
$$q_F(\tilde \alpha_i,\tilde \alpha_j) =<\alpha_i,\gamma^2>\cdot <\alpha_j,\gamma^2> -  <\alpha_i,\alpha_j>.$$
where $<  , >$ is the symmetric bilinear form on $H^4(X,\Q)$. Suppose that $F \simeq S^{[2]}$, with $S$ a K3 surface.Then the homomorphism $H^2(S,\Q) \to H^2(F,\Q)$ induces an orthogonal direct sum decomposition with respect to the Beauville-Bogomolov form
\begin {equation} H^2(F,\Q) \simeq H^2(S,\Q) \oplus \Q \delta,\end{equation}
with $q_F(\delta,\delta)=-2$ and $q_F$ restricted to $H^2(S,\Q)$ is the intersection form, see [SV,(56)]. \par  
\noindent  A cubic fourfold $X$ is special if contains a surface $\Sigma$ such that its cohomological class $\sigma$ in $H^4(X,\Q)$ is not homologous to any multiple of  $\gamma^2$. Therefore $\rho_2(X) >1$. The discriminant $d$ is defined as the discriminant  of the intersection form $< ,>_D$ on the subspace $D$ of $H^4(X,\Q)$ generated by $\sigma$ and $\gamma^2$. By a result of B.Hasset in [Has 1, 1.0.3],  if $X$ is a generic special cubic fourfold with discriminant of the form $d =2(n^2 +n+1)$, where $n$ is an integer $\ge2$, then the Fano variety of $X$ is isomorphic to $S^{[2]}$, with $S$ a K3 surface associated to X.  
Such special $X$ include cubic fourfolds containing a plane, or  a  cubic scroll, or a Veronese surface, see [Has 1, 4.1]. Special cubic fourfolds of discriminant $d$ form a nonempty irreducible divisor $\sC_d$ in the moduli space $\sC$ of cubic four folds if and only if  $d>0$ and $d\equiv 0,2 (6)$. If $d$ is not divisible by 4,9 or any odd prime $p \equiv 2 (3)$ then $X \in \sC_d$ if and only if the transcendental lattice  $\sT_X$ is Hodge isometric  to $\sT_S(-1)$ for some K3 surface, see [Add].\par
\noindent If  $X$ is special and $F(X) \simeq S^{[2]}$, with $S$ a K3 surface, then by (4.5),  $H^2_{tr}(F,\Q)\simeq H^2_{tr}(S,\Q)$, where $\dim H^2_{tr}(F,\Q) = \dim H^4_{tr}(X,\Q) =23 -\rho_2(X)$. Here $\rho_2(X) \ge2$ and hence we get
$$ \dim H^2_{tr}(F,\Q) =\dim H^2_{tr}(S,\Q) = 22 -\rho(S) = 23 -\rho_2(X) \le 21,$$
where $\rho(S)$ is the rank of $NS(S)$. \par

 \begin{thm}  Let $X$ be  a cubic fourfold and let $F =F(X)$ be the Fano variety of lines. Suppose that   $F \simeq S^{[2]}$, with  $S$ a K3 surface.  Let  $p$ and $q$ be the morphisms in 
 the incidence diagram  (4.4). Then $q$ induces a map of motives $\bar q : t_2(S)(1) \to t(X)$ in $\sM_{rat}(\C)$ which is  an isomorphism.  
\end{thm}
\begin{proof}  In (4.4) $P$ is a $\P^1$-bundle over $F$ (see [Vois 4, 3.24] ) and hence
$$ h(P) \simeq h(F) \oplus h(F)(1)$$
 By the results in [deC-M, 6.2.1] $h(S)(1)$ is a direct summand of $h(S^{[2]})=h(F)$. Therefore we get a map of motives
\begin {equation} \CD   \bar q : h(S)(1)@>{f}>>h(F)@>{g}>>h(P)@>{q_*}>>h(X) \\ \endCD\end{equation}
where $f$  and $g$ are  the inclusions induced by the splittings of $h(F)$ and $h(P)$.  Let $t_2(S)$ be the transcendental motive of the surface $S$. By composing with the inclusion $t_2(S)(1) \to h(S)(1)$ and the surjection $h(X) \to t(X)$ we get from (4.7) a map of motives  in  $\sM_{\rat }(\C)$,
$$ \bar q :  t_2(S)(1) \to t(X)$$ 
The morphisms  $p$ and $q$ yield a homomorphism  $\Psi_0=p^*q_*: A^i(F) \to A^{i-1}(X)$ 
such that 
$$ \Psi_0 : A^4(F) =A_0(F) \to A^3(X)=A_1(X),$$
is surjective, see [TZ, 6.1(i)].  By [SV,15.6] the group $A_0(F)$ decomposes as follows
$$A_0(F) =A_0(F)_0 +A_0(F)_2 +A_0(F)_4$$
where $A_0(F)_0 =\Q c_F$, with  $c_F \in A_0(F)$ a special degree 1 cycle. For two distinct points $x,y \in S$ let's denote by $[x,y] \in F =S^{[2]}$ the point of $F$ that corresponds to the subscheme $\{x,y\} \subset S$. If  $x=y$  then $[x,x]$ denotes the element in $A_0(F)$ represented by any point corresponding to a non reduced subscheme of length 2 on $S$ supported on $x$. With these notations $c_F$ is represented by the point $[c_S, c_S]  \in F$, where $c_S$  is the Beauville-Voisin cycle in $A_0(S)$ such that $c_2(S) = 24  c_S$. We also have (see [SV,15.6]) 
$$A_0(F)_2 = <[c_S,x]- [c_S,y]>,$$
We claim  that the map $\phi : A_0(S) \to A_0(S^{[2]})=A_0(F)$ sending $[x]$ to $[c_S, x]$ is injective and hence $A_0(S)_0$ is isomorphic to $A_0(F)_2 \subset A_0(F)_{hom}$.\par
 \noindent The variety $S^{[2]}$ is the blow-up of the symmetric product $S^{(2)}$ along $S$. Let $\tilde S$ be the inverse image of $S$ in $S^{[2]}$. Then $\tilde S$ is the image of the closed embedding 
$s \to [c_S, s]$. By a result proved in [Ba] the induced map of 0-cycles $A_0( \tilde S) \to A_0(S^{[2]})$ is injective. Therefore the map $\phi$  is injective .\par
\noindent By [SV, 20.3 ] we have 
$$A_0(F)_4 =A^2(F)_2 \cdot A^2(F)_2$$ 
where $A^2(F)_2) =\sA_{hom}$ and  $\sA$  is generated by the cycles of the form $S_x $ with $x \in S$, see [ SV, 15.2]. Here, for a point $x \in S$,  the variety $S_x  \subset F$ is isomorphic to the blow up of $S$ at the point $X$ and $S_x \cdot S_y =[x,y]$. Then, using the results in [SV, 20.2 (iii)] and [SV, 20.5], we get
 $$\ker (\Psi _0 : A^4(F)  \to A^3(X))= F^4A_0(F)= \sA_{hom}\cdot \sA_{hom}=A_0(F)_4.$$
 Therefore $\Psi_0$ induces  an exact sequence 
 $$ A_0(F)_4 \to   A_0(F)_{hom} \to A_1(X)_{hom} \to 0$$
 which yields an isomorphism between $A_0(S)_0=A_0(F)_2\subset A_0(F)_{hom}$ and $A_1(X)_{hom}$.\par
 \noindent  We have $A^i(t_2(S)) =0$ for $i \ne 2$ and $A^2(t_2(S))= A_0(t_2(S))=A_0(S)_0$. Therefore
 we get an isomorphism $A^3(t_2(S)(1))=A^2(t_2(S)) =A_0(S)_{hom} \simeq A_1(X)_{hom}$. From Lemma 4.2 we get $A^i(t(X) =0$, for $ i\ne 3$ and $A^3(t(X) =A_1(X)_{hom}$. \par
 \noindent Therefore the map of motives $\bar q : M \to N$ in $\sM_{rat}(\C)$, where $M=t_2(S)(1)$ and $N =t(X)$, induces isomorphisms $A^i(M) \simeq A^i(N)$ for all $i \ge0$. Then, from Lemma 4.3 , 
 $\bar q$ is an isomorphism.
  \end{proof}
\begin{rk} Let $X$ be a cubic fourfold such that there exist K3 surfaces $S_1$ and $S_2$ and  isomorphisms $r_1 : F(X) \to S^{[2]}_1$ and $r_2: F(X) \to S^{[2]}_2$ with  $r_1^*\delta_1 \ne r^*_2\delta_2$, as in 
[Has 1, 6.2.1]. Then by Theorem 4.6 $t_2(S_1) \simeq t_2(S_2)$, and hence the motives $h(S_1)$ and $h(S_2)$ are isomorphic.\end{rk}
\begin{cor} Let $X$ be  a cubic fourfold and let $F =F(X)$ be the Fano variety of lines. Suppose that   $F \simeq S^{[2]}$, with  $S$ a K3 surface. Then $h(X)$ is finite dimensional   if and only if   $h(S)$ is finite dimensional in which case the motive $t(X)$ is indecomposable.
 \end{cor}
 \begin{proof} If $h(X)$ is finite dimensional then also $t(X)$ is finite dimensional and hence, by Theorem 4.6, $t_2(S)$ is finite dimensional . Therefore $h(S)$ is finite dimensional. Conversely, if $h(S)$ is finite dimensional then also $t_2(S)$ and $t(X)$ are finite dimensional, by Theorem 4.6.  From the Chow-K\"unneth decomposition in (4.1) we get that $h(X)$ is finite dimensional.   If $h(S)$ is finite dimensional  then the motive $t_2(S)$ is indecomposable, see [Vois 1, Cor. 3.10], and hence also $t(X)$ is indecomposable.  \end{proof}
\begin {rk} If the motive $h(X)$ is finite dimensional then $t(X)$ is, up to isomorphisms in $\sM_{rat}(\C)$, independent of the Chow-K\"unneth decomposition  $h(X) =\sum_i h_i(X)$ in (4.1). If   $h(X) =\sum_i \tilde h_i(X)$ is another Chow-K\"unneth decomposition, with $\tilde h_i(X) =(X,\tilde \pi_i)$, then, by [KMP, 7.6.9], there is an isomorphism $\tilde h_i(X) \simeq h_i(X)$ and 
$\tilde  \pi_i =(1+Z)\circ \pi_i \circ (1+Z)^{-1}$, where $Z \in A^4(X \times X)_{hom}$ is a nilpotent correspondence. In particular 
$$\tilde \pi_4= (1+Z)\circ \pi_4\circ (1+Z)^{-1}=(1+Z)\circ (\pi^{alg}_4+\pi^{tr}_4) \circ (1+Z)^{-1}$$
and hence $\tilde h_4(X)$ contains as a direct summand a submotive $\tilde t(X) =(X, (1+Z)\circ \pi^{tr}_4 \circ (1+Z)^{-1})$ isomorphic to $t(X)$.\par
\noindent However, differently from the case of the transcendental motive $t_2(S)$ of a surface $S$, the motive $t(X)$ is not a birational invariant. In fact $t(X)\ne0$ for a rational cubic fourfold $X$ such that $F(X) \simeq S^{[2]}$, with $S$ a K3 surface
\end{rk}
 \begin {ex} (1) A family  $\sF$ of cubic fourfolds  $X$ with a finite dimensional motive is given by the equations
$$ u^3 +v^3 +f(x,y,z,t)=0$$
where $(x,y,z,t,u,v)$ are coordinates in $\P^5$,  $f$ is of degree 3 and defines a smooth surface in $\P^3$, see [Lat 1, Rk. 18]. For each fourfold $X \in \sF$  there is an associated  K3 surface $S$ (a double cover of $\P^2$ ramified along a sextic) and a correspondence $\Gamma \in A^3(S \times X) = \End_{\sM_{rat}}(h(S)(1),h(X))$ inducing an isomorphism $ \Gamma_*: A_0(S)_0\simeq A_1(X)_{hom}=A_1(X)_{alg}$, see [Vois 3, 4.2]. By Lemma 4.2 the correspondence $\Gamma$ induces a map of motives $\gamma : t_2(S)(1) \to t(X)$ such that $\gamma_* : A^i(t_2(S)(1)) \to A^i(t(X)$ is an isomorphism for all $i\ge0$.Therefore, by Lemma 4.3, $\gamma$ is an isomorphism of motives. The motives of all four folds $X$ in $\sF$ are finite dimensional and of abelian type, see [Lat 1,  Cor.17(iii)],and hence also the motives $t_2(S)$ are finite dimensional and of abelian type. It follows that  all the associated K3 surfaces $S$ have a finite dimensional motive of abelian type.\par
\noindent (2) Examples of rational cubic fourfolds $X$ such that the transcendental motive $t(X)$ is isomorphic to $t_2(S)(1)$, for a given K3 surface $S$ can be constructed as in [Has 3, Sect.5]. Let $A^1(S)$ be generated by two ample divisors $h_1$ and $h_2$ such that $<h_1, h_1> = <h_2, h_2> =  2$ and $<h_1,h_2>= <h_2,h_1> = 2k +1$, where $k =2,3$ and $<,>$ is the intersection form on $H^2(S,\Z)$. Then there is diagram

$$ \CD  Y@>{\phi}>> X \subset \P^5 \\
       @V{\pi}VV            @.            \\
       \P^2 \times \P^2    \\ \endCD $$
      
\noindent where  $Y$ is smooth,  $\pi$ is the blow up of a surface $S'$, the image of $(s_1,s_2 ) : S \to \P^2 \times \P^2 $ where   $s_i : S \to  \P^2$ is the branched  double cover given by $\sO_{S}(h_i)$. The map $\phi$ is obtained   by blowing up a plane $P$ and surface  $T$ on $X$. If $k=2$ then  also  $T$  is a  plane, while $T$ is a Veronese surface in the case $k=3$.  In both cases we have $t_2(P)=t_2(T)=0$. The motive $h(Y)$ splits as  $h(S')(1) \oplus h(\P^2 \times \P^2)$, where  $\P^2 \times \P^2$ has no transcendental motive, and $h(X)$ is a direct summand of $h(Y)$. Therefore the transcendental motive $t(X)$ coincides with
$t_2(S')(1)$ and $t_2(S') =t_2(S)$, because $t_2(-)$ is a birational invariant.\par
(3) Let  $\sC _{14}$ be the irreducible divisor of special  cubic fourfolds of degree14, for which the special surface is a smooth quartic rational normal scroll. By the results in [BR] all the fourfolds $X$ in $\sC_{14}$are rational. Moreover if $X \in (\sC_{14} - \sC_8) $, then $F(X) \simeq S^{[2]}$, where $S$ is the K3 surface of degree 14 and genus 8 parametrizing smooth quartic rational normal scrolls contained in $X$. By Theorem 4.6 there is an isomorphism between the transcendental motives $t(X) \simeq t_2(S) (1)$.\end{ex}
According to Corollary 4.9 if $X$ is a special cubic fourfold with $F(X) \simeq S^{[2]}$, and $h(X)$ is  finite dimensional, then $t(X)$ is indecomposable. The following proposition shows that, if $X$ is very general  and $h(X)$ is  finite dimensional, then  $t(X)$ is indecomposable.  
\begin {prop} Let $X$ be a very general cubic fourfold. If  $h(X)$ is finite dimensional  the transcendental motive $t(X)$ is indecomposable.\end{prop}
\begin {proof} Let the primitive motive $h(X)_{prim}=( X,\pi_{prim},0)$ be defined as in [Ki, 8.4], where
$$ \pi_{prim} = \Delta_X - (1/3)\sum_{0 \le i \le 4}( \gamma^ {4-i} \times \gamma^i).$$
and
$$ H^*(h(X)_{prim})= H^4(X,\Q)_{prim} = P(X)_{\Q}$$
\noindent If $X$ is very general then  $\rho_2(X)=1$ and $ A^2(X)$ is generated by the class $\gamma^2$. Therefore in the Chow-K\"unneth decomposition of $h(X)$ in (4.1)  we have
$h(X)_{prim} = h^{tr}_4(X) =t(X)$ and
$$h_4(X) =h^{alg}_4(X) +h^{tr}_4(X) \simeq \L \oplus h(X)_{prim}.$$
If  $X$ is very general, then  $\End_{HS} (H^4(X,\Q)_{prim} )=\Q[id]$, see [Vois 2, Lemma 5.1]. Let $\sM_{hom}(\C)$ be the category of homological motives and let $\tilde \sM_{hom}(\C)$ be the subcategory generated by the motives of all smooth projective varieties  $V$ such that the K\"unneth components of the diagonal in $H^*(V \times V)$ are algebraic.  The Hodge realization functor
$$H_{Hodge} : \sM_{rat}(\C) \to HS_{\Q}$$
to the Tannakian category of $\Q$-Hodge structures induces a faithful functor $\tilde \sM_{hom}(\C) \to HS_{\Q}$. Let $\bar h(X)= h^{hom}(X) \in \tilde \sM_{hom}(\C)$ : then
$\End_{\sM_{hom}}(\bar h(X)_{prim}) \simeq \Q[id]$ and hence
$$\End_{\sM_{hom}}(\bar h^{tr}_4((X)) \simeq  \End_{\sM_{hom}}(\bar h(X)_{prim}) \simeq \Q[id]$$
 If $h(X)$ is finite dimensional then the indecomposability of $\End_{\sM_{hom}}(\bar h^{tr}_4((X))$ in $\sM_{hom}(\C)$ implies the indecomposability in $\sM_{rat}(\C)$. Therefore  
 $$\End_{\sM_{rat}}(t(X)) \simeq \End_{\sM_{rat}}(h(X)_{prim}) \simeq \Q[id]$$
and the transcendental motive of $X$ is  indecomposable.\par
\end{proof}

  \end{document}